\documentclass[12pt,reqno]{amsproc}
\usepackage{times}

\usepackage{amsmath,amsthm,amssymb,wasysym,mathtools}
\usepackage[inline]{enumitem}
\usepackage{caption}
\usepackage{subcaption}
\usepackage{graphicx}

\DeclareMathOperator{\RE}{Re}

\numberwithin{equation}{section}

\usepackage{amsthm}

\newtheoremstyle{fancytheorem}               
{.5\baselineskip±.2\baselineskip}           
{.5\baselineskip±.2\baselineskip}           
{\itshape\addtolength{\leftskip}{0mm}\setlength{\parindent}{0em}}      
{0mm}                                     
{\bfseries}                                 
{}                                  
{.5em}                                      
{\thmname{#1}\thmnumber{ #2}. \thmnote{[#3]}}

\newtheoremstyle{fancydefinition}               
{.5\baselineskip±.2\baselineskip}           
{.5\baselineskip±.2\baselineskip}           
{\upshape\addtolength{\leftskip}{0mm}\setlength{\parindent}{0em}}      
{0mm}                                     
{\bfseries}                                 
{}                                  
{.5em}                                      
{\thmname{#1}\thmnumber{ #2}. \thmnote{[#3]}}

\theoremstyle{fancytheorem}

\newtheorem{theorem}{Theorem}[section]
\newtheorem{lemma}[theorem]{Lemma}

\theoremstyle{fancydefinition}

\allowdisplaybreaks
\usepackage{xcolor} 

\begin{document}
	
	\title[Product of starlike  functions with non-vanishing  polynomials]{Starlikeness of a product of starlike \\ functions with non-vanishing  polynomials}

	\author[S. Malik]{Somya Malik}
	\address{Department of Mathematics \\National Institute of Technology\\Tiruchirappalli-620015,  India }
	\email{arya.somya@gmail.com}
	
	\author[V. Ravichandran]{V. Ravichandran} 
	\address{Department of Mathematics \\National Institute of Technology\\Tiruchirappalli-620015,  India }
	\email{vravi68@gmail.com; ravic@nitt.edu}
	
	\begin{abstract}
		For a function $f$  starlike of order $\alpha$, $0\leqslant \alpha <1$,   a non-constant polynomial $Q$ of degree $n$ which is non-vanishing in the unit disc $\mathbb{D}$ and $\beta>0$, we consider the function $F:\mathbb{D}\to\mathbb{C}$ defined by $F(z)=f(z) (Q(z))^{\beta /n}$ and  find the largest value of $r\in (0,1]$ such that $r^{-1} F(rz)$ lies in various known subclasses of starlike functions such as the class of starlike functions of order $\lambda$, the classes of starlike functions associated with the exponential function, cardioid, a rational function, nephroid domain and modified sigmoid function. Our radii results are sharp.  We also discuss the correlation with known radii results as special cases.
	\end{abstract}

	
	\thanks{The first author is supported by the UGC-JRF Scholarship.}
	
	\maketitle
	
	\section{Introduction}We consider the class $\mathcal{A}$ of all analytic functions defined on the unit disc $\mathbb{D}:=\{z: |z|<1\}$ and normalised by the condition $f(0)=f'(0)-1=0$ as well as its subclass $\mathcal{S}$ consisting of all univalent functions. For any two subclasses $\mathcal{F}$ and $\mathcal{G}$ of $\mathcal{A}$, the $\mathcal{G}$-radius of the class $\mathcal{F}$, denoted as $R_{\mathcal{G}} (\mathcal{F})$ is the largest number $R_{\mathcal{G}} \in (0,1)$ such that $r^{-1}f(rz)\in \mathcal{G}$ for all $f\in \mathcal{F}$ and $0<r<R_{\mathcal{G}}$. In 1969, Ba\c{s}g\"{o}ze \cite{Bas} studied the radii of starlikeness and convexity  for polynomial functions which are non-vanishing in the unit disc $\mathbb{D}$. This study was motivated by the work of Alexander \cite{Alex} who showed that the radius of starlikeness and hence the radius of univalence for the function $f$ defined by $f(z)=zP(z)$, where $P$ is a polynomial of degree $n>0$ with zeros outside the unit disc, is $(n+1)^{-1}$. Further, Ba\c{s}g\"{o}ze \cite{Bas2} also studied the radius results related to $\alpha$-spirallike and $\alpha$- convex spirallike functions of order $\lambda$ for various kinds of functions obtained from polynomials, such as $zP(z),P(z)^{\beta/n}, z(P(z))^{\beta/n}, zM(z)/N(z)$ where $P,M,\ \text{and}\ N$ are polynomials with zeros outside the unit disc. In 2000,   Gangadharan et al.\ \cite{Ganga2}  (see also Kwon and Owa \cite{KwanOwa}) determined the radius of $p$-valently strongly starlikeness of order $\gamma$ for the function $F:\mathbb{D}\to\mathbb{C}$ defined by  $F(z):=f(z) (Q(z))^{\beta/n}$, where $f$ is a $p$-valent analytic function, $Q$ has properties similar to that of the polynomials considered in the paper by Ba\c{s}g\"{o}ze, and $\beta$ is a positive real number. The present study continues this investigation to several recently studied subclasses of starlike functions defined by subordination.
	
	For two analytic functions $f$ and $g$, $f$ is said to be subordinate to $g$, denoted by $f\prec g$, if $f=g\circ w$ for some  analytic function $w:\mathbb{D}\rightarrow \mathbb{D}$, with $w(0)=0$. When the function $g$ is univalent, the subordination $f\prec g$ holds if and only if $f(0)=g(0)$ and $f(\mathbb{D})\subseteq g(\mathbb{D})$. Several subclasses of the class $\mathcal{A}$ are defined using the concept of subordination. For an analytic function $\varphi:\mathbb{D}\to\mathbb{C}$, the class $\mathcal{S}^{*}(\varphi)$  consists of all  functions $f\in \mathcal{A}$ satisfying the subordination \[\dfrac{zf'(z)}{f(z)}\prec \varphi(z).\]
	Shanmugam \cite{Shan}   studied convolution theorems for  more general classes but with stronger assumption of convexity of $\varphi$. 	Ma and Minda \cite{MaMinda} later gave a unified treatment of growth, distortion, rotation and coefficient inequalities for the class $\mathcal{S}^{*}(\varphi)$ when the superordinate function $\varphi$ is a function with positive real part, $\varphi(\mathbb{D})$ is symmetric with respect to the real axis and starlike with respect to $\varphi(0)=1$.  	The class of starlike functions of order $\alpha$, $\mathcal{ST}(\alpha)$ is a special case when $\varphi(z)=(1+(1-2\alpha)z)/(1-z)$; for $\alpha=0$, the usual class of starlike functions is obtained. Other subclasses of the class of starlike functions can also be derived for different choices of $\varphi$. For $\varphi_{par}(z)=1+(2/\pi^2) (\log ((1+\sqrt{z})/(1-\sqrt{z})) )^2$, the class $\mathcal{S}_{p}=\mathcal{S}^{*}(\varphi_{par})$ is the class of starlike functions associated with a parabola, the image of the unit disc under the function $\varphi_{par}$ is the set $\varphi_{par}(\mathbb{D})=\{w=u+\iota v: v^2<2u-1\}=\{w: |w-1|<\RE w\}$. This class was introduced by  R\o nning \cite{RonPara}. Mendiratta et al. \cite{Exp} investigated the class of starlike functions associated with the exponential function $\mathcal{S}^{*}_\mathit{e}=\mathcal{S}^{*}(\mathit{e}^z)$. Similarly, various properties of the class of starlike functions associated with a cardioid, $\mathcal{S}^{*}_{c}=\mathcal{S}^{*}(\varphi_{c}(z))$ for $\varphi_{c}(z)=1+(4/3)z+(2/3)z^2$ are studied by Sharma et al.\cite{KanNavRavi}. A function $f\in \mathcal{S}^{*}_{c}$ if $zf'(z)/f(z)$ lies in the region bounded by the cardioid $\{x+\iota y: (9x^2+9y^2-18x+5)^2 -16(9x^2+9y^2-6x+1)=0\}$. Kumar and Ravichandran \cite{KumarRavi} discussed the class $\mathcal{S}^{*}_{R}=\mathcal{S}^{*}(\psi(z))$ of starlike functions associated with a rational function for $\psi(z)=1+((z^2+kz)/(k^2-kz)),\ k=\sqrt{2}+1$. In 2020, Wani and Swaminathan \cite{LatSwami} studied the class of starlike functions associated with a nephroid domain, $\mathcal{S}^{*}_{Ne}=\mathcal{S}^{*}(\varphi_{Ne})$ with $\varphi_{Ne}(z)=1+z-z^3/3$. Goel and Kumar \cite{PriyankaSivaSig} explored various properties of the class $\mathcal{S}^{*}_{SG}=\mathcal{S}^{*}(2/(1+\mathit{e}^{-z}))$ known as the class of starlike functions associated with modified sigmoid function. Radius problems relating to the ratio of analytic functions are recently considered in \cite{LeckoRaviAsha,Madhuravi}. In 2020, Wani and Swaminathan \cite{LatSwami2} discussed the radius problems for the functions associated with the nephroid domain. Cho and others have also investigated some interesting radius problems, see \cite{ChoVirendSushRavi,EbadianCho}.

	For a given   function $f$   starlike  of order $\alpha$, $Q$  a  polynomial of degree $n$  whose zeros are outside the unit disk $\mathbb{D}$ (in other words, $Q$ is non-vanishing in $\mathbb{D}$) and $\beta$  a positive real number, we determine the $\mathcal{M}$-radius of the function  $F:\mathbb{D}\to\mathbb{C}$ defined by  $F(z):=f(z) (Q(z))^{\beta/n}$ for various choices of the class $\mathcal{M}$.  In Section 2, we determine  the radius of starlikeness of order $\lambda$ for the function $F$, and obtain, as a special case,   the radius of starlikeness for $F$. This is done by studying a mapping property of the function $w:\mathbb{D}\to\mathbb{C}$ defined by $w(z):=zF'(z)/F(z)$. Indeed,   we find the smallest disc containing $w(\overline{\mathbb{D}}_r)$   where $\mathbb{D}_r:=\{z\in \mathbb{C}:|z|<r\}$.  This disc is used in the investigation of all the radius problems.  In Sections 3-4, we respectively compute the values of the radius of starlikeness associated with the exponential function and the radius of starlikeness associated with a cardioid for the function $F$. In Sections 5-7, we determine the radius of starlikeness associated with a particular rational function, the radius of starlikeness associated with nephroid domain  and the radius of starlikeness associated with modified sigmoid function for the function $F$. All the radii obtained are shown to be sharp. Several known radii results are obtained as special cases for specific values of $\alpha$ and $\beta$.
	
	\section{Starlike functions of order $\lambda$}
	For $0\leqslant \lambda <1$, the class  $\mathcal{ST}(\lambda)$ of starlike functions of order $\lambda$ contains all functions $f\in \mathcal{A}$, satisfying $\RE (zf'(z)/f(z))>\lambda$. Let $f\in \mathcal{ST}(\alpha)$ and $Q$ be a polynomial of degree $n>0$ with all of its zeros in the region $\mathbb{C}\backslash\mathbb{D}$. Since $Q$ is non-vanishing in $\mathbb{D}$, the function $F:\mathbb{D}\rightarrow \mathbb{C}$   defined by
	\begin{equation}\label{eqn2.1}
		F(z):=f(z) (Q(z))^{\frac{\beta}{n}},\quad \beta>0
	\end{equation}is analytic in $\mathbb{D}$. We assume $Q(0)=1$ throughout this paper so that $F\in\mathcal{A}$.
	The following theorem gives the $\mathcal{ST}(\lambda)$-radius for the function $F$ and it is independent of the degree of the polynomial $Q$.

	\begin{theorem}\label{thm2.1}If the function $f\in \mathcal{ST}(\alpha)$,  then the radius of starlikeness of order $\lambda$ for the function $F$ defined in \eqref{eqn2.1} is given by \[R_{\mathcal{ST}(\lambda)} (F)=\dfrac{2(1-\lambda)}{2-2\alpha+\beta+\sqrt{(2-2\alpha+\beta)^2+4(1-\lambda)(2\alpha+\beta-1-\lambda)}}.\]
	\end{theorem}
	\begin{proof}
		We start with finding the disc in which $zF'(z)/F(z)$ lies for $z\in \overline{\mathbb{D}}_{r}$, then using this disc, we determine the radius of starlikeness of order $\lambda$ for $F$.
		
		For the function $F$ given by \eqref{eqn2.1}, a calculation shows that
		\begin{equation}\label{eqn2.2}
			\dfrac{zF'(z)}{F(z)}=\dfrac{zf'(z)}{f(z)}+\dfrac{\beta}{n}\dfrac{zQ'(z)}{Q(z)}.
		\end{equation}
		Since $f\in \mathcal{ST}(\alpha)$, it is well-known that $zf'(z)/f(z)$ has positive real part and so
		\begin{equation}\label{eqn2.3a}
			\left| \frac{zf'(z)}{f(z)}-\dfrac{1+(1-2\alpha)r^2}{1-r^2}\right| \leqslant \dfrac{2(1-\alpha)r}{1-r^2},\quad |z|\leqslant r.
		\end{equation}
		or equivalently
		$zf'(z)/f(z)$ lies in the disc $\mathbb{D} (a_f (r); c_f(r))$ for $|z|\leqslant r$ where
		\begin{equation}\label{eqn2.3}
			a_f(r)=\dfrac{1+(1-2\alpha)r^2}{1-r^2},\quad c_f(r)=\dfrac{2(1-\alpha)r}{1-r^2}.
		\end{equation}
		Let $z_{k},\ k=1,2,\ldots ,n$ denote the zeros of the polynomial $Q$, then the polynomial $Q$ is a constant multiple of $\prod_{k=1}^n (z-z_k)$ and so
		\begin{equation}\label{eqn2.4}
			\dfrac{zQ'(z)}{Q(z)}=\sum_{k=1}^{n}\dfrac{z}{z-z_{k}}.
		\end{equation}
		Since $z_{k}\in \mathbb{C}\backslash \mathbb{D}$ for every $k$, the bilinear transformation $z/(z-z_k)$ maps $\overline{\mathbb{D}}_r$ to a disc. Indeed,    \cite[Lemma~3.2]{GangaRaviShan} shows that \begin{equation*}
			\left|\dfrac{z}{z-z_{k}}+\dfrac{r^2}{1-r^2}\right|\leqslant \dfrac{r}{1-r^2},\quad |z|\leqslant r
		\end{equation*}  for every  $k$ and hence, using \eqref{eqn2.4}, we have
		\begin{equation}\label{eqn2.5}
			\left|\dfrac{zQ'(z)}{Q(z)}+\dfrac{nr^2}{1-r^2}\right|\leqslant \dfrac{ nr}{1-r^2},\quad |z|\leqslant r.
		\end{equation}
		Using \eqref{eqn2.2}, we get
		\begin{align}
			\left|\dfrac{zF'(z)}{F(z)}-\dfrac{1-(2\alpha -1+\beta)r^2}{1-r^2}\right| &=\nonumber \left|\dfrac{zf'(z)}{f(z)}-\dfrac{1+(1-2\alpha )r^2}{1-r^2}+\dfrac{\beta}{n}\dfrac{zQ'(z)}{Q(z)}+\dfrac{\beta r^2}{1-r^2}\right|\\
			&\leqslant \left|\dfrac{zf'(z)}{f(z)}-\dfrac{1+(1-2\alpha )r^2}{1-r^2}\right|+\left|\dfrac{\beta}{n}\dfrac{zQ'(z)}{Q(z)}+\dfrac{\beta r^2}{1-r^2}\right|,\label{eqn2.6}
		\end{align}
		for $|z|\leqslant r$.
		Since $\beta\in \mathbb{R}$ is positive, using the equations \eqref{eqn2.3a} and \eqref{eqn2.5}, the inequality   \eqref{eqn2.6} gives
		\begin{equation}\label{eqn2.7}
			\left|\dfrac{zF'(z)}{F(z)}-\dfrac{1-(2\alpha -1+\beta)r^2}{1-r^2}\right| \leqslant \dfrac{(2-2\alpha+\beta)r}{1-r^2},\quad  |z|\leqslant r.
		\end{equation}
		Define the functions $a_F $ and $c_F$ by
		\[ a_F(r):= \dfrac{1-(2\alpha -1+\beta)r^2}{1-r^2}  \quad \text{and}\quad c_F(r):= \dfrac{(2-2\alpha+\beta)r}{1-r^2}. \]
		so that $ zF'(z)/F(z)\in  \mathbb{D}(a_F (r);c_F(r))$.
		It is observed that the center $a_{F}$ is an increasing function of $r$ for $2\alpha+\beta-2<0$, and is a decreasing function of $r$ for $2\alpha+\beta-2 \geqslant 0$.

		From \eqref{eqn2.7}, it follows that
		\begin{align}
			\RE \dfrac{zF'(z)}{F(z)} &\geqslant \notag a_{F}(r)-c_{F}(r)\\
			&= \notag \dfrac{1-(2\alpha -1+\beta)r^2}{1-r^2}-\dfrac{(2-2\alpha+\beta)r}{1-r^2}\\
			&=\dfrac{1-(1-2\alpha)r}{1+r}-\dfrac{\beta r}{1-r}=: \psi(r) . \label{eqn2.10}
		\end{align}
		The equation
		$\psi(\alpha, \beta, r)=\lambda$ is simplifies to
		\begin{equation}\label{eqn2.9}
			(1-2\alpha-\beta+\lambda)r^2-(2-2\alpha+\beta)r+1-\lambda=0
		\end{equation} and so the smallest positive root of the equation $\psi(\alpha, \beta, r)=\lambda$
		in the interval $(0,1)$ is given by  \[\sigma_{0}:=\dfrac{2(1-\lambda)}{2-2\alpha+\beta+\sqrt{(2-2\alpha+\beta)^2+4(1-\lambda)(2\alpha+\beta-1-\lambda)}}.\]
		It can be seen that \
		\begin{align}
			\psi '(r)&=\nonumber - \dfrac{2(1-\alpha)(1-r)^2+\beta(1+r)^2}{(1-r^2)^2}<0,
		\end{align}
		which shows that $\psi$ is a decreasing function of $r\in (0,1)$. Therefore,
		for $r<\sigma_{0}$, we have $		\psi(\alpha,\beta,r)>\psi(\alpha,\beta,\sigma_{0})=\lambda$ and so \eqref{eqn2.10} implies that $ \RE (zF'(z)/F(z))\geqslant \psi(\alpha,\beta,r)>\lambda$ for all $r< \sigma_{0}$, or in other words, the  radius of starlikeness of order $\lambda$  of the function $F$ is at least $\sigma_{0}$.
		
		To show that the radius obtained is the best possible, take $f(z)=z(1-z)^{2\alpha-2}\in \mathcal{ST}(\alpha)$ and the polynomial $Q(z)=(1+z)^{n}$. For these choices of the functions $f$ and $Q$, we have
		\begin{align}
			F(z)&= \notag z(1-z)^{2\alpha-2} (1+z)^{\beta},\\
			\intertext{which implies}
			\dfrac{zF'(z)}{F(z)}&=\notag 1+\dfrac{(2-2\alpha)z}{1-z}+\dfrac{\beta z}{1+z}\\
			&=\notag \dfrac{(1-2\alpha-\beta)z^2+(2-2\alpha+\beta)z+1}{1-z^2}\\
			&=\lambda+ \dfrac{(1-2\alpha-\beta+\lambda)z^2+(2-2\alpha+\beta)z+1-\lambda}{1-z^2}.\label{eqn2.11}
		\end{align}
		Using the fact that $\sigma_{0}$ is the positive root of the polynomial in \eqref{eqn2.9}, the equation \eqref{eqn2.11} shows that  $\RE (zF'(z)/F(z))=\lambda$ for $z=-\sigma_{0}$ proving sharpness of $\sigma_0$.
	\end{proof}

	For $\lambda =0$, the radius of starlikeness for the function $F$ given by \eqref{eqn2.1} is  \[R_{\mathcal{ST}}(F)=\dfrac{2}{2-2\alpha+\beta+\sqrt{(2-2\alpha+\beta)^2+4(2\alpha+\beta-1)}}.\]
	When $\alpha$ goes to $1$ in Theorem~\ref{thm2.1} we get that the radius of starlikeness of order $\lambda$ for the function $F(z)=z(Q(z))^{\beta/n}$, where $Q$ is a non-constant polynomial of degree $n$ non-vanishing on the unit disc and $\beta >0$, comes out to be $(1-\lambda)/(\beta+1-\lambda)$. This result for $\beta=n$ coincides with the one obtained by Ba\c{s}g\"{o}ze in \cite[Theorem~3]{Bas}.	Moreover,  by letting  $\beta\to 0$  in Theorem~\ref{thm2.1},  we obtain  the radius of starlikeness of order $\lambda$ for the class of starlike functions of order $\alpha, 0\leqslant \alpha <1$ (see \cite[p.~88]{Goodman}).
	
	\section{Starlike functions associated with the exponential function}
	The class $\mathcal{S}^{*}_{\mathit{e}}$ of starlike functions associated with the exponential functions consists of all the functions $f\in \mathcal{A}$ which satisfy  $ zf'(z)/f(z)\prec \mathit{e}^z$. This class was introduced and studied  by Mendiratta et al. \cite{Exp} in 2015. The subordination definition is equivalent to the inequality $|\log (zf'(z)/f(z))| <1$.
	
	The main result of the section provides the $\mathcal{S}^{*}_{\mathit{e}}$ radius for the function $F$ given by \eqref{eqn2.1}.
	\begin{lemma}\cite[Lemma~2.2]{Exp}\label{lem4.1}
		For $1/\mathit{e}<a<\mathit{e}$, let $r_a$ be given by
		\begin{align*}
			r_a &=\begin{dcases}
				a-\frac{1}{\mathit{e}} &\text{ if }\ \frac{1}{\mathit{e}}< a \leqslant \frac{\mathit{e}+\mathit{e}^{-1}}{2}\\
				\mathit{e}-a &\text{ if }\ \frac{\mathit{e}+\mathit{e}^{-1}}{2} \leqslant a <\mathit{e}.
			\end{dcases}
		\end{align*}
		Then,  $\{w: |w-a|<r_a\} \subset \{w: |\log w|<1\}= \Omega_{\mathit{e}}$, where $\Omega_{\mathit{e}}$ is the image of the unit disc $\mathbb{D}$ under the exponential function.
	\end{lemma}
	
	\begin{theorem}\label{thm4.2}
		If the function $f\in \mathcal{ST}(\alpha)$, then the radius of starlikeness asscociated with the exponential function for the function $F$ defined in \eqref{eqn2.1} is given by
		\begin{align*}
			R_{\mathcal{S}^{*}_{\mathit{e}}}(F)&= \begin{dcases}
				\sigma_{0} &\text{ if }\ 2\alpha+\beta-2\geqslant 0\\
				\sigma_{0} &\text{ if }\ 2\alpha+\beta-2<0\ \text{and}\ X(\alpha,\beta)\leqslant 0\\
				\tilde{\sigma_{0}}&\text{ if }\ 2\alpha+\beta-2<0\ \text{and} X(\alpha,\beta)> 0,
			\end{dcases}
		\end{align*}
		where
		\begin{align}
			\sigma_{0}=& \nonumber\dfrac{2(\mathit{e}-1)}{\mathit{e}(2-2\alpha+\beta)+\sqrt{(\mathit{e}(2-2\alpha+\beta))^2-4(\mathit{e}-1)(1-\mathit{e}(2\alpha-1+\beta))}},\\
			\tilde{\sigma_{0}} =& \nonumber \dfrac{2(\mathit{e}-1)}{(2-2\alpha+\beta)+\sqrt{(2-2\alpha+\beta)^2-4(\mathit{e}-1)(2\alpha-1+\beta-\mathit{e})}}
			\intertext{and}
			X(\alpha,\beta) =& \label{eqn4.1} 2(2+\beta-2\alpha)(1+\mathit{e}^2-2\mathit{e}(2\alpha+\beta -1))((-2\alpha+2+\beta
			)\\ &\quad \nonumber-\sqrt{(-2\alpha+2+\beta)^2-4(\mathit{e}-1)(2\alpha-1+\beta-\mathit{e})})\\
			&\quad\nonumber +4(2\alpha+\beta-1-\mathit{e})(\mathit{e}^2-1)(2\alpha+\beta-2).
		\end{align}
	\end{theorem}
	\begin{proof}
		Our aim is to show that the $\mathbb{D}(a_F (r);c_F(r)) \subset \Omega_{\mathit{e}}$ for all $0<r\leqslant R_{\mathcal{S}^{*}_{\mathit{e}}}(F)$. Let
		\begin{align*}
			\sigma_{0}:=&\dfrac{2(\mathit{e}-1)}{\mathit{e}(2-2\alpha+\beta)+\sqrt{(\mathit{e}(2-2\alpha+\beta))^2-4(\mathit{e}-1)(1-\mathit{e}(2\alpha-1+\beta))}},\intertext{and}
			\tilde{\sigma_{0}} :=&\dfrac{2(\mathit{e}-1)}{(2-2\alpha+\beta)+\sqrt{(2-2\alpha+\beta)^2-4(\mathit{e}-1)(2\alpha-1+\beta-\mathit{e})}}.
		\end{align*}
		It is clear that $\sigma_{0}$ and $\tilde{\sigma_{0}}$ are both positive as both $2-2\alpha+\beta$ and $\mathit{e}-1$ are positive..
		For the polynomial
		\begin{equation}\label{eqn4.2}
			\phi(r):=(1-\mathit{e}(-1+2\alpha+\beta))r^2-\mathit{e}(2-2\alpha+\beta)r+(\mathit{e}-1)
		\end{equation} obtained from the equivalent form $\phi(r)=0$ of the equation $c_{F}(r)=a_{F}(r)-1/\mathit{e}$, it is observed that $\phi(0)=\mathit{e}-1>0,\ \phi(1)=-2\mathit{e} \beta<0 $, showing that there is a zero for $\phi$ in the interval $(0,1)$, namely $\sigma_{0}$. Also, the positive root of the equation $c_{F}(r)=\mathit{e}-a_{F}(r)$ or $\psi(r)=0$ with
		\begin{equation}\label{eqn4.3}
			\psi(r):=(-1+2\alpha+\beta-\mathit{e})r^2-(2-2\alpha+\beta)r+(\mathit{e}-1)
		\end{equation} is $\tilde{\sigma_{0}}$. To verify that $\psi$ has a zero in the interval $(0,1)$, it is seen that $\psi(0)=\mathit{e}-1>0$ and $\psi(1)=4(\alpha-1)<0$.
		
		Case (i): $2\alpha+\beta-2\geqslant 0$. Since the center $a_{F}$ in \eqref{eqn2.7} is a decreasing function of $r$, $a_{F}(r)>a_{F}(\sigma_0),\ \text{for}\ r\in (0,\sigma_0)$. By definition, $\sigma_0$ is the solution the equation $c_{F}(r)=a_{F}(r)-1/\mathit{e}$, also, the radius in \eqref{eqn2.7} satisfies $c_{F}(r)>0,\ r\in (0,1)$, together imply that $a_{F}(r)>a_{F}(\sigma_{0})>1/\mathit{e},\ \text{for}\  r\in (0,\sigma_{0})$. Further,  $a_{F}(r)< a_{F}(0)=1<(\mathit{e}+\mathit{e}^{-1})/2\approx1.54308,\ \text{for}\ r\in (0,\sigma_0) $. Thus, it is established that \[\dfrac{1}{\mathit{e}}<a_{F}(r)\leqslant \dfrac{\mathit{e}+\mathit{e}^{-1}}{2\mathit{e}},\ \text{for}\ r\in (0,\sigma_0).\]
		Applying Lemma~\ref{lem4.1} we get $\mathbb{D}(a_F (\sigma_0);c_F(\sigma_0)) \subset \Omega_{\mathit{e}}$ that is, the radius of starlikeness associated with the exponential function for the function $F$ is atleast $\sigma_{0}$.
		
		Case (ii): $2\alpha+\beta -2<0$ and $X(\alpha,\beta)\leqslant 0$. The number \[\tilde{\sigma_{1}}=\sqrt{\dfrac{1-2\mathit{e}+\mathit{e}^2}{1+2\mathit{e}+\mathit{e}^2-4\mathit{e} \alpha-2\mathit{e} \beta}}<1\]
		is the positive root of the equation $a_{F}(r)=(\mathit{e}+\mathit{e}^{-1})/2$, or equivalently, $\zeta(r)=0$ with  $\zeta(r):=(1+\mathit{e}^2-2\mathit{e}(2\alpha+\beta-1))r^2+2\mathit{e}-\mathit{e}^2-1$. Here, $\zeta(0)=2\mathit{e}-\mathit{e}^2-1\approx -2.9525<0$ and $\zeta(1)=-2\mathit{e}(2\alpha+\beta-2)>0$ justifies the existence of a zero for the polynomial $\zeta$ in the interval $(0,1)$. Further,
		\begin{align}\label{eqn4.4}
			\zeta(\tilde{\sigma_{0}}) =& \dfrac{1}{4(-1-\mathit{e}+2\alpha+\beta)^2} \big (
			2(2+\beta-2\alpha)(1+\mathit{e}^2-2\mathit{e}(2\alpha+\beta -1))\\&\quad\nonumber \times((-2\alpha+2+\beta
			) -\sqrt{(-2\alpha+2+\beta)^2-4(\mathit{e}-1)(2\alpha-1+\beta-\mathit{e})}) \\\nonumber&\quad{}+ 4(2\alpha+\beta-1-\mathit{e})(\mathit{e}^2-1)(2\alpha+\beta-2)\big),
		\end{align}
		and from this with \eqref{eqn4.1} we infer that $X(\alpha,\beta) \leqslant 0$ implies $\zeta(\tilde{\sigma_{0}}) \leqslant 0$. Also, $\zeta(0)<0,\ \zeta(\tilde{\sigma_{0}})\leqslant 0$ along with the fact $\zeta(\tilde{\sigma_{1}})=0$ gives $\tilde{\sigma_{0}} \leqslant \tilde{\sigma_{1}}$, which implies that $a_{F}(\tilde{\sigma_{0}})\leqslant a_{F}(\tilde{\sigma_{1}})=(\mathit{e}+\mathit{e}^{-1})/2$. The application of Lemma~\ref{lem4.1} gives $\mathbb{D}(a_F (\sigma_0);c_F(\sigma_0)) \subset \Omega_{\mathit{e}}$ or in other words, the radius of starlikeness associated with the exponential function for $F$ is atleast $\sigma_{0}$.
		
		Case (iii): $2\alpha+\beta -2<0$ and $X(\alpha,\beta)> 0$. Here following the same line of thought as in Case (ii),  $\zeta(\tilde{\sigma_{0}})>0$ implies $a_{F}(\tilde{\sigma_{0}})>(\mathit{e}+\mathit{e}^{-1})/2$, and thus  Lemma~\ref{eqn4.1} gives the required radius to be atleast $\tilde{\sigma_{0}}$.
		
		To show that the obtained radius values are the best possible, take $f(z)=z/(1-z)^{-2\alpha+2}\in \mathcal{ST}(\alpha)$ and the polynomial $Q(z)=(1+z)^{n}$, these choices give
		the expression for $zF'(z)/F(z)$, as already shown in the proof of Theorem~\ref{thm2.1}, to be
		\begin{align}
			\dfrac{zF'(z)}{F(z)} &=\dfrac{(1-2\alpha-\beta)z^2+(1+(1-2\alpha)+\beta)z+1}{1-z^2}\nonumber\\
			&=\mathit{e}-\dfrac{(-1+2\alpha+\beta-\mathit{e})z^2-(2-2\alpha+\beta)z+\mathit{e}-1}{1-z^2}.\label{eqn4.6}
			\intertext{It is seen that \eqref{eqn4.6} can also be written as} \dfrac{zF'(z)}{F(z)}&=\dfrac{1}{\mathit{e}}+\dfrac{(1-\mathit{e}(-1+2\alpha+\beta))z^2+\mathit{e}(2-2\alpha+\beta)z+\mathit{e}-1}{\mathit{e}(1-z^2)}.\label{eqn4.7}
		\end{align}
		The definition of the polynomial $\phi$ in \eqref{eqn4.2} for $r=\sigma_{0}$ together with \eqref{eqn4.7} gives \[\left|\log \dfrac{(-\sigma_{0})F'(-\sigma_{0})}{F(-\sigma_{0})}\right|=\left|\log \dfrac{1}{\mathit{e}}\right|=1,\]
		proving sharpness for $\sigma_{0}$.
		Further, the polynomial $\psi$ in \eqref{eqn4.3} for $r=\tilde{\sigma_{0}}$ and \eqref{eqn4.6} provide
		\[\left|\log \dfrac{\tilde{\sigma_{0}}F'(\tilde{\sigma_{0}})}{F(\tilde{\sigma_{0}})}\right|=\left|\log \mathit{e}\right|=1.\]
		This proves sharpness for $\tilde{\sigma_{0}}$.
	\end{proof}
	
	If we let $\alpha$ goes to $1$ in Theorem~\ref{thm4.2}, then the radius of starlikeness associated with the exponential function for the function $F(z)=z(Q(z))^{\beta/n}$ where $Q$ is a non-constant polynomial of degree $n$ non-vanishing on the unit disc and $\beta >0$, comes out to be $(\mathit{e}-1)/(\mathit{e}\beta+\mathit{e}-1)$.	Moreover, when $\beta\to 0$ in Theorem~\ref{thm4.2}, we obtain  the radius of starlikeness associated with the exponential function for the class of starlike functions of order $\alpha, 0\leqslant \alpha <1$ obtained by Mendiratta et al. in \cite[Theorem~3.4]{Exp} and also by  Khatter et al. in \cite[ Theorem~2.17 (2)]{KhatSivaRaviExpo}    for $\mathcal{S}^{*}_{0,\mathit{e}}$  with $A=1-2\alpha\ \text{and}\  B=-1$.
	\section{Starlike functions associated with a cardioid}
	Sharma et al. \cite{KanNavRavi} studied the class $\mathcal{S}^{*}_{c}=\mathcal{S}^{*}(\varphi_{c}),\ \varphi_{c}(z)=(1+(4/3)z+(2/3)z^2)$ of starlike functions associated with a cardioid and proved the following lemma.
	\begin{lemma}\cite[lemma~2.5]{KanNavRavi} \label{lem5.1}
		For $1/3<a<3$,
		\begin{align*}
			r_a &=
			\begin{dcases}
				\frac{3a-1}{3} & \text{ if }\ \frac{1}{3}<a\leq \frac{5}{3}\\
				3-a & \text{ if } \ \frac{5}{3}\leq a<3.
			\end{dcases}
		\end{align*}
		Then $\{w: |w-a|<r_a\} \subset \Omega _c$. Here $\Omega_c$ is the region bounded by the cadioid $\{x+\iota y: (9x^2+9y^2-18x+5)^2 -16(9x^2+9y^2-6x+1)=0\}$.
	\end{lemma}
	The following theorem gives the radius of starlikeness associated with a cardioid for the function $F$ given in \eqref{eqn2.1}.
	\begin{theorem}\label{thm5.2}
		Let the function $f\in \mathcal{ST}(\alpha)$, the $\mathcal{S}^{*}_{c}$ radius for the function $F$ defined in \eqref{eqn2.1} is
		\begin{align*}
			R_{\mathcal{S}^{*}_{c}}(F)&= \begin{dcases}
				\sigma_{0} &\text{ if }\ 2\alpha+\beta-2\geqslant 0\\
				\sigma_{0} &\text{ if }\ 2\alpha+\beta-2<0\ \text{and}\ X(\alpha,\beta)\leqslant 0\\
				\tilde{\sigma_{0}}&\text{ if }\ 2\alpha+\beta-2<0\ \text{and} X(\alpha,\beta)> 0,
			\end{dcases}
		\end{align*}
		where
		\begin{align}
			\sigma_{0}=&\nonumber \dfrac{4}{3(2-2\alpha+\beta)+\sqrt{(6-6\alpha+3\beta)^2+8(6\alpha+3\beta-4)}},\\
			\tilde{\sigma_{0}} =& \nonumber\dfrac{4}{(2-2\alpha+\beta)+\sqrt{(2-2\alpha+\beta)^2-8(2\alpha+\beta-4)}}
			\intertext{and}
			X(\alpha,\beta) =& \label{eqn5.1} 2(8-6\alpha-3\beta)(6\alpha-6-3\beta)(6\alpha-6-3\beta \\&\quad\nonumber+ \sqrt{(6-6\alpha+3\beta)^2+8(6\alpha+3\beta-4)})\\&\quad\nonumber+ 48(6\alpha+3\beta-4)(2-2\alpha-\beta).
		\end{align}

	\end{theorem}
	
	\begin{proof}
		Consider the disc mentioned in \eqref{eqn2.7}, it needs to be shown that this disc lies in the region $\Omega_{c}$ for all $0<r\leqslant R_{\mathcal{S}^{*}_{c}}(F)$. Take the equations $c_{F}(r)=a_{F}(r)-1/3$ and $c_{F}(r)=3-a_{F}(r)$; which are equivalent to $\phi(r)=0\ \text{and}\ \psi(r)=0$ respectively, with the corresponding polynomials $\phi$ and $\psi$ given by
		\begin{align}
			\phi(r):&=(3(2\alpha-1+\beta)-1)r^2+3(2-2\alpha+\beta)r-2,\label{eqn5.2}\intertext{and}
			\psi(r):&= (2\alpha-1+\beta-3)r^2-(2-2\alpha+\beta)r+2.\label{eqn5.3}
		\end{align}
		For the polynomial $\phi$, $\phi(0)=-2<0,\ \phi(1)=6\beta>0$; thus $\phi$ has a zero in the interval $(0,1)$, let it be denoted by $\sigma_{0}$. Also, considering the polynomial $\psi$, it is seen that $\psi(0)=2>0$ and $\psi(1)=4(\alpha-1)<0$, let the zero of $\psi$ in the interval $(0,1)$ be denoted by $\tilde{\sigma_{0}}$. Then it is obtained that
		\begin{align*}
			\sigma_{0}:=&\dfrac{4}{3(2-2\alpha+\beta)+\sqrt{(6-6\alpha+3\beta)^2+8(6\alpha+3\beta-4)}},\intertext{and}
			\tilde{\sigma_{0}} :=&\dfrac{4}{(2-2\alpha+\beta)+\sqrt{(2-2\alpha+\beta)^2-8(2\alpha+\beta-4)}}.
		\end{align*}
		Here, it is evident from their values that both $\sigma_{0}$ and $\tilde{\sigma_{0}}$ are indeed positive.

		Case (i): $2\alpha+\beta-2 \geqslant 0$. First it is shown that the center $a_{F}$ in \eqref{eqn2.7} satisfies
		\begin{equation}\label{eqn5.4}
			\dfrac{1}{3}<a_{F}(r)< \dfrac{5}{3},\ r\in (0,\sigma_{0}).
		\end{equation}
		The fact that the center is a decreasing function of $r$, implies $a_{F}(r)>a_{F}(\sigma_{0})\ \text{for}\  r\in (0,\sigma_0)$. Also, $\sigma_{0}$ is the root of the equation $a_{F}(r)-1/3=c_{F}(r)$, along with the fact that the radius satisfies $c_{F}(r)>0,\ r\in (0,1)$ gives $a_{F}(r)>a_{F}(\sigma_{0})>1/3\ \text{for}\ r\in (0,\sigma_{0})$. Further, $a_{F}(r)<a_{F}(0)=1<5/3,\ r\in (0,\sigma_0)$. This  proves \eqref{eqn5.4} and applying Lemma~\ref{lem5.1}, infers that $\mathbb{D}(a_F (\sigma_{0});c_F(\sigma_{0})) \subset \Omega_{c}$, that is, the radius of starlikeness associated with a cardioid for the function $F$ is atleast $\sigma_{0}$.
		
		Case (ii): $2\alpha+\beta-2<0\ \text{and}\ X(\alpha,\beta)\leqslant 0$. The equation $a_{F}(r)=5/3$ from Lemma~\ref{lem5.1}, takes the form $\zeta(r)=0$, with $\zeta(r):=(8-6\alpha-3\beta)r^2-2$. Then, it is seen that $\zeta(0)=-2<0$ and $\zeta(1)=-3(2\alpha+\beta-2)>0$. This shows that $\zeta$ has a zero in the interval $(0,1)$, let this be denoted by $\tilde{\sigma_{1}}$, then \[\tilde{\sigma_{1}}=\sqrt{\dfrac{2}{8-6\alpha-3\beta}}.\]
		Also since
		\begin{align}
			\zeta(\sigma_{0}) =&\label{eqn5.4a} \dfrac{1}{4(-4+6\alpha+3\beta)^2}(2(8-6\alpha-3\beta)(6\alpha-6-3\beta)(6\alpha-6-3\beta \\
			&\quad{}\nonumber+ \sqrt{(6-6\alpha+3\beta)^2 +8(6\alpha+3\beta-4)}) \\&\quad\nonumber+48(6\alpha+3\beta-4)(2-2\alpha-\beta)),
		\end{align}
		from \eqref{eqn5.1} and \eqref{eqn5.4a} it is evident that $X(\alpha,\beta)\leqslant 0$ is equivalent to saying $\zeta(\sigma_{0})\leqslant 0$. The facts that $\zeta(0)<0,\ \zeta(\sigma_{0})\leqslant 0$ and $\zeta(\tilde{\sigma_{1}})=0$ together imply that $\sigma_{0} \leqslant \tilde{\sigma_{1}}$.  Further, $\sigma_{0}\leqslant \tilde{\sigma_{1}}$ will imply that $a_{F}(\sigma_{0})\leqslant a_{F}(\tilde{\sigma_{1}})=5/3$, and thus using Lemma~\ref{lem5.1}, $\mathbb{D}(a_F (\sigma_{0});c_F(\sigma_{0})) \subset \Omega_{c}$. Thus, proving that the required radius value is atleast $\sigma_{0}$.
		
		Case (iii):  $2\alpha+\beta-2<0\ \text{and}\ X(\alpha,\beta)> 0$. On the similar lines as in Case (ii), $X(\alpha,\beta)>0$ implies $\zeta(\sigma_{0})>0$, which gives that $a_{F}(\sigma_{0})>5/3$, this inturn, after another application of  Lemma~\ref{lem5.1} concludes that the radius of starlikeness associated with a cardioid for the function $F$ is atleast $\tilde{\sigma_{0}}$.

		To verify the sharpness of the obtained radius values, take $f(z)=z/(1-z)^{-2\alpha+2}\in \mathcal{ST}(\alpha)$ and the polynomial $Q$ as $Q(z)=(1+z)^{n}$. Thus, the expression for $zF'(z)/F(z)$ as seen in the proof for Theorem~\ref{thm2.1}, becomes
		\begin{align}
			\dfrac{zF'(z)}{F(z)}&= \dfrac{(1-2\alpha-\beta)z^2+(2-2\alpha+\beta)z+1}{1-z^2}.\label{eqn5.5}
		\end{align}
		The polynomial $\phi$ in \eqref{eqn5.2} gives that for $r=\sigma_{0}$,
		\begin{equation}
			3((1-2\alpha-\beta)r^2-(2-2\alpha+\beta)r+1)=1-r^2. \label{eqn5.6}
		\end{equation}
		Thus, the sharpness for $\sigma_{0}$ is proved by using \eqref{eqn5.5} and \eqref{eqn5.6} which gives that $zF'(z)/F(z)=1/3=\varphi_{c}(-1)$ for $z=-\sigma_{0}$. 	Also, the polynomial $\psi$ in \eqref{eqn5.3} for $r=\tilde{\sigma_{0}}$ gives
		\begin{equation}
			(1-2\alpha-\beta)r^2+(2-2\alpha+\beta)r+1=3(1-r^2). \label{eqn5.7}
		\end{equation}
		Thus, using \eqref{eqn5.7} in \eqref{eqn5.5} it is seen that $\tilde{\sigma_{0}}F'(\tilde{\sigma_{0}})/F(\tilde{\sigma_{0}})=3=\varphi_{c}(1)$. This proves the sharpness of $\tilde{\sigma_{0}}$.
	\end{proof}
	
	If we let $\alpha$ goes to $1$ in Theorem~\ref{thm5.2}, then we obtain that the radius of starlikeness associated with a cardioid for the function $F(z)=z(Q(z))^{\beta/n}$, where $Q$ is a non-constant polynomial of degree $n$ non-vanishing on the unit disc and $\beta >0$, comes out to be $2/(2+3\beta)$. Moreover,  by letting $\beta\to 0$   in Theorem~\ref{thm5.2}, we get the radius of starlikeness associated with a cardioid (\cite[Theorem~4.7]{KanNavRavi} with $A=1-2\alpha\ \text{and}\  B=-1$) for the class of starlike functions of order $\alpha, 0\leqslant \alpha <1$.
	
	\section{Starlike functions associated with a rational function}
	Kumar and Ravichandran \cite{KumarRavi} introduced the class of starlike functions associated with the rational function $\varphi_{R}(z)=1+((z^2+kz)/(k^2-kz)),\ \text{with}\  k=\sqrt{2}+1$. This class is represented by  $\mathcal{S}^{*}_{R}=\mathcal{S}^{*}(\varphi_{R}(z))$.They also showed the following result, which is used in finding the $\mathcal{S}^{*}_{R}$ radius for the function $F$ defined in \eqref{eqn2.1}.
	
	\begin{lemma}\cite[lemma~2.2]{KumarRavi} \label{lem6.1}
		For $2(\sqrt{2}-1)<a<2,$
		\begin{align*}
			r_a &=
			\begin{dcases}
				a-2(\sqrt{2}-1) & \text{ if }\ 2(\sqrt{2}-1)<a\leq \sqrt{2}\\
				2-a & \text{ if } \ \sqrt{2}\leq a<2.
			\end{dcases}
		\end{align*}
		Then $\{w: |w-a|<r_a\} \subset \varphi_{R}(\mathbb{D}).$
	\end{lemma}
	
	\begin{theorem}\label{thm6.2}
		If the function $f\in \mathcal{ST}(\alpha)$, then the $\mathcal{S}^{*}_{R}$ radius for the function $F$ defined in \eqref{eqn2.1} is given by
		\begin{align*}
			R_{\mathcal{S}^{*}_{R}}(F)&= \begin{dcases}
				\sigma_{0} &\text{ if }\ 2\alpha+\beta-2\geqslant 0\\
				\sigma_{0} &\text{ if }\ 2\alpha+\beta-2<0\ \text{and}\ X(\alpha,\beta)\leqslant 0\\
				\tilde{\sigma_{0}}&\text{ if }\ 2\alpha+\beta-2<0\ \text{and} X(\alpha,\beta)> 0,
			\end{dcases}
		\end{align*}
		where
		\begin{align}
			\sigma_{0}=&\nonumber\dfrac{2(3-2\sqrt{2})}{(2-2\alpha+\beta)+\sqrt{(-2+2\alpha-\beta)^2-4(3-2\sqrt{2})(2\sqrt{2}-1-2\alpha-\beta)}},\\
			\tilde{\sigma_{0}} =&\nonumber\dfrac{2}{(2-2\alpha+\beta)+\sqrt{(-2+2\alpha-\beta)^2-4(2\alpha-3+\beta)}}
			\intertext{and}
			X(\alpha,\beta) =& \label{eqn6.1}2(2\alpha-2-\beta)(1+\sqrt{2}-2\alpha-\beta)((2\alpha-2-\beta)\\
			&\quad \nonumber+\sqrt{(2\alpha-2-\beta)^2-4(3-2\sqrt{2})(2\sqrt{2}-1-2\alpha-\beta)}) \\
			&\quad \nonumber
			+4(1-2\sqrt{2}+2\alpha+\beta) ((3-2\sqrt{2})(1+\sqrt{2}-2\alpha-\beta) \\
			&\quad \nonumber +(1-\sqrt{2})(1-2\sqrt{2}+2\alpha+\beta)).
		\end{align}
	\end{theorem}
	\begin{proof}
		We show that the disc mentioned in \eqref{eqn2.7} satisfies $\mathbb{D}(a_F (r);c_F(r)) \subset \varphi_{R}(\mathbb{D})$ for all $0<r\leq R_{\mathcal{S}^{*}_{R}}(F)$. Lemma~\ref{lem6.1} gives that the two possible values of the radius are the smallest positive roots of the equations $c_{F}(r)=a_{F}(r)-2(\sqrt{2}-1)$ and $c_{F}(r)=2-a_{F}(r)$; which are equivalent to $\phi(r)=0$ and $\psi(r)=0$ respectively, where the polynomials in $r$ are of the form
		\begin{align}
			\phi(r):&=(2(\sqrt{2}-1)-(2\alpha-1+\beta))r^2-(2-2\alpha+\beta)r+3-2\sqrt{2},\label{eqn6.2}\intertext{and}
			\psi(r):&=(2\alpha-3+\beta)r^2-(2-2\alpha+\beta)r+1\label{eqn6.3}
		\end{align}
		respectively. The fact that both the polynomials $\phi$ and $\psi$ possess zeros in the interval $(0,1)$ can be easily verified as $\phi(0)=3-2\sqrt{2}>0 \ \text{and}\ \phi(1)=-2\beta <0$, also, $\psi(0)=1>0,\  \psi(1)=4(\alpha-1)<0$. The respective positive zeros of $\phi$ and $\psi$, denoted by $\sigma_{0}$ and $\tilde{\sigma_{0}}$, are given by
		\begin{align*}
			\sigma_{0}:=&\dfrac{2(3-2\sqrt{2})}{(2-2\alpha+\beta)+\sqrt{(-2+2\alpha-\beta)^2-4(3-2\sqrt{2})(2\sqrt{2}-1-2\alpha-\beta)}},\intertext{and}
			\tilde{\sigma_{0}} :=&\dfrac{2}{(2-2\alpha+\beta)+\sqrt{(-2+2\alpha-\beta)^2-4(2\alpha-3+\beta)}}.
		\end{align*}
		
		Case (i): $2\alpha+\beta-2 \geqslant 0$. By using the fact that the center $a_{F}(r)$ is decreasing function of $r$, it will be proved that
		\begin{equation}\label{eqn6.4}
			2(\sqrt{2}-1)<a_{F}(r)\leq \sqrt{2},\ r\in (0,\sigma_{0}).
		\end{equation}
		This after the application of Lemma~\ref{eqn6.1} will directly imply$\mathbb{D}(a_F (\sigma_{0});c_F(\sigma_{0})) \subset \varphi_{R}(\mathbb{D})$, that is, the required radius is atleast $\sigma_{0}$. So, to prove \eqref{eqn6.4}, observe that $ a_{F}(r)>a_{F}(\sigma_{0})\  \text{for}\  r\in (0,\sigma_0)$, and $\sigma_{0}$ is the positive root of the equation $c_{F}(r)=a_{F}(r)-2(\sqrt{2}-1)$, also since the radius $c_{F}(r)$ in \eqref{eqn2.7} is positive for all $r\in (0,1)$, $a_{F}(\sigma_{0})-2(\sqrt{2}-1)>0$. Lastly, $a_{F}(r)<a_{F}(0)=1<\sqrt{2},\ r\in (0,\sigma_0)$, thus proving \eqref{eqn6.4}, and also the required result for this case.
		
		Case (ii): $2\alpha+\beta-2<0\ \text{and}\ X(\alpha,\beta)\leqslant 0$. Here, again from the Lemma~\ref{lem6.1}, consider the equation $a_{F}(r)=\sqrt{2}$, which gets simplified to $\zeta(r)=0$, with $\zeta(r)=(\sqrt{2}+1-2\alpha-\beta)r^2+1-\sqrt{2}$. Then, for the polynomial $\zeta$, the positive root is denoted by $\tilde{\sigma_{1}}$, where
		\begin{equation*}
			\tilde{\sigma_{1}}=\sqrt{\dfrac{\sqrt{2}-1}{\sqrt{2}+1-2\alpha-\beta}}.
		\end{equation*}
		It is observed that $\zeta(1)=-(2\alpha+\beta -2)>0$ and $\zeta(0)=1-\sqrt{2}<0$, justifying the fact that $\tilde{\sigma_{1}} \in (0,1)$. This also gives that $\sigma_{0} \leqslant \tilde{\sigma_{1}}$ if and only if $\zeta(\sigma_{0})\leqslant 0$. It is seen that
		\begin{align}
			\zeta(\sigma_{0})=&\label{eqn6.4a}\dfrac{1}{4(2\sqrt{2}-1-2\alpha-\beta)^2}\big( 2(2\alpha-2-\beta)(1+\sqrt{2}-2\alpha-\beta)((2\alpha-2-\beta)\\
			&\quad \nonumber+\sqrt{(2\alpha-2-\beta)^2-4(3-2\sqrt{2})(2\sqrt{2}-1-2\alpha-\beta)}) \\
			&\quad \nonumber
			+4(1-2\sqrt{2}+2\alpha+\beta) ((3-2\sqrt{2})(1+\sqrt{2}-2\alpha-\beta) \\
			&\quad \nonumber +(1-\sqrt{2})(1-2\sqrt{2}+2\alpha+\beta))\big),
		\end{align}
		so, combining \eqref{eqn6.1} and \eqref{eqn6.4a}, $X(\alpha,\beta)\leqslant 0$ is same as saying $\zeta(\sigma_{0})\leqslant 0$. Thus, in this case, $\sigma_{0} \leqslant \tilde{\sigma_{1}}$ which implies $a_{F}(\sigma_{0})\leqslant a_{F}(\tilde{\sigma_{1}})=\sqrt{2}$, and now, Lemma~\ref{lem6.1} gives that the radius of starlikeness associated with the rational function for the function $F$ is atleast $\sigma_{0}$.
		
		Case (iii):  $2\alpha+\beta-2<0\ \text{and}\ X(\alpha,\beta)> 0$. Following similar arguments as in Case (ii), $\zeta(\sigma_{0})>0$ gives $a_{F}(\sigma_{0})> \sqrt{2}$, and then Lemma~\ref{lem6.1} implies that the required radius is atleast $\tilde{\sigma_{0}}$.

		Take $f(z)=z/(1-z)^{-2\alpha+2}\in \mathcal{ST}(\alpha)$ and the polynomial $Q$ as $Q(z)=(1+z)^{n}$. Using these choices for the function $f$ and the polynomial $Q$, the expression for $zF'(z)/F(z)$ as in the proof for Theorem~\ref{thm2.1}, becomes,
		\begin{align}
			\dfrac{zF'(z)}{F(z)}=& \dfrac{(1-2\alpha-\beta)z^2+(2-2\alpha+\beta)z+1}{1-z^2}.\label{eqn6.5}
		\end{align}
		The polynomials $\phi$ and $\psi$ in \eqref{eqn6.2}, \eqref{eqn6.3} for $r=\sigma_{0}$ and $r=\tilde{\sigma_{0}}$  respectively imply that
		\begin{align}
			(1-2\alpha-\beta)r^2-(2-2\alpha+\beta)r+1 &=2(\sqrt{2}-1)(1-r^2), \label{eqn6.6}\intertext{and}
			(1-2\alpha-\beta)r^2 +(2-2\alpha+\beta)r+1 &=2(1-r^2).\label{eqn6.7}
		\end{align}
		Now, observe that using \eqref{eqn6.6} and putting $z=-\sigma_{0}$ in \eqref{eqn6.5}, it is obtained that \[
		\dfrac{(-\sigma_{0})F'(-\sigma_{0})}{F(-\sigma_{0})}= 2(\sqrt{2}-1)=\varphi_{R}(-1)\]
		this proves the sharpness for the radius $\sigma_{0}$. Also, considering \eqref{eqn6.7}, and replacing $z=\tilde{\sigma_{0}}$ in \eqref{eqn6.5}, it is seen that $zF'(z)/F(z)$ assumes the value $2=\varphi_{R}(1)$ thus proving the sharpness for $\tilde{\sigma_{0}}$.
	\end{proof}
	
	When $\alpha$ goes to $1$ in Theorem~\ref{thm6.2} we get that the radius of starlikeness associated with a rational function for the function $F(z)=z(Q(z))^{\beta/n}$ where $Q$ is a non-constant polynomial of degree $n$ non-vanishing on the unit disc and $\beta >0$, comes out to be $(3-2\sqrt{2})/(3-2\sqrt{2}+\beta)$. 	Further, when $\beta\to 0$,  we obtain the radius of starlikeness associated with a rational function for the class of starlike functions of order $\alpha, 0\leqslant \alpha <1$ obtained by   Kumar and Ravichandran in \cite[Theorem~3.2]{KumarRavi} (when $A=1-2\alpha\ \text{and}\  B=-1$).
	
	\section{Starlike functions associated with a nephroid domain}
	Wani and Swaminathan \cite{LatSwami} in 2020 studied the class of starlike functions associated with a nephroid domain $\mathcal{S}^{*}_{Ne}=\mathcal{S}^{*}(\varphi_{Ne})$ with $\varphi_{Ne}(z)=1+z-z^3/3$. The function $\varphi_{Ne}$ maps the unit disc onto  the interior of the nephroid, a 2-cusped curve, \[\left((u-1)^2+v^2-\dfrac{4}{9}\right)^3-\dfrac{4v^2}{3}=0.\]
	In this section, the radius of starlikeness associated with the nephroid is discussed for the function $F$ defined in \eqref{eqn2.1}, using the following lemma due to Wani and Swaminathan.
	\begin{lemma}\cite[lemma~2.2]{LatSwami2} \label{lem7.1}
		For $1/3<a<5/3,$
		\begin{align*}
			r_a &=
			\begin{dcases}
				a-\frac{1}{3} & \text{ if }\ \frac{1}{3}<a\leq 1\\
				\frac{5}{3}-a & \text{ if } \ 1\leq a<\frac{5}{3}.
			\end{dcases}
		\end{align*}
		Then $\{w: |w-a|<r_a\} \subset \Omega_{Ne}$, where $\Omega_{Ne}$ is the region bounded by the nephroid, that is \[\Omega_{Ne}:=\left\{\left((u-1)^2+v^2-\dfrac{4}{9}\right)^3-\dfrac{4v^2}{3}<0\right\}.\]
	\end{lemma}
	
	\begin{theorem}\label{thm7.2}
		If the function $f\in \mathcal{ST}(\alpha)$, then the radius of starlikeness associated with a nephroid domain for the function $F$ defined in \eqref{eqn2.1} is given by
		\begin{align*}
			R_{\mathcal{S}^{*}_{Ne}}(F)&= \begin{dcases}
				\sigma_{0} &\text{ if }\ 2\alpha+\beta-2\geqslant 0\\
				\tilde{\sigma_{0}} &\text{ if }\ 2\alpha+\beta-2<0,
			\end{dcases}
		\end{align*}
		where
		\begin{align*}
			\sigma_{0}&=\dfrac{4}{3(2-2\alpha+\beta)+\sqrt{9(-2+2\alpha-\beta)^2-8(4-6\alpha-3\beta)}},\intertext{and}
			\tilde{\sigma_{0}} &=\dfrac{4}{3(2-2\alpha+\beta)+\sqrt{9(-2+2\alpha-\beta)^2-8(6\alpha-8+3\beta)}}.
		\end{align*}
	\end{theorem}
	\begin{proof}
		The proof aims to show that the disc in \eqref{eqn2.7} satisfies the following condition:\[\mathbb{D}(a_F (r);c_F(r)) \subset \Omega_{Ne},\quad 0<r\leqslant R_{\mathcal{S}^{*}_{Ne}}(F).\]
		
		Let \[\sigma_{0}:=\dfrac{4}{3(2-2\alpha+\beta)+\sqrt{9(-2+2\alpha-\beta)^2-8(4-6\alpha-3\beta)}}.\]
		The number $\sigma_{0}$ is the positive solution to the equation $c_{F}(r)=a_{F}(r)-1/3$, which transforms into $\phi(r)=0$ with the polynomial $\phi$ in $r$ given by
		\begin{equation}\label{eqn7.1}
			\phi(r):=(4-6\alpha-3\beta)r^2-3(2-2\alpha+\beta)r+2.
		\end{equation}
		For the polynomial $\phi$, $\phi(0)=2>0$ and $\phi(1)=-6\beta <0$ justifying the existence of a zero for $\phi$ in the interval $(0,1)$.
		The number \[\tilde{\sigma_{0}}: =\dfrac{4}{3(2-2\alpha+\beta)+\sqrt{9(-2+2\alpha-\beta)^2-8(6\alpha-8+3\beta)}},\]
		is  the positive zero of the polynomial $\psi$ in $r$ is given by
		\begin{equation}\label{eqn7.2}
			\psi(r):=(6\alpha-8+3\beta)r^2-3(2-2\alpha+\beta)r+2.
		\end{equation}
		Here, the equation $\psi(r)=0$ is the the simplified form of the equation $c_{F}(r)=5/3-a_{F}(r)$. The polynomial $\psi$ indeed has a zero in the interval $(0,1)$ as $\psi(0)=2>0$ and $\psi(1)=12(\alpha-1)<0$.
		It is known that the center $a_{F}$ in \eqref{eqn2.7} has the property that $a_{F}(r)\leqslant 1$ for $2\alpha+\beta-2\geqslant 0$; and $a_{F}(r)>1$ for $2\alpha+\beta-2<0$.
		
		Case (i): $2\alpha+\beta-2\geqslant 0$. In this case, the center $a_{F}(r)\leqslant 1$, thus Lemma~\ref{lem7.1} implies that $\mathbb{D}(a_F (\sigma_{0});c_F(\sigma_{0})) \subset \Omega_{Ne}$. This proves that the $\mathcal{S}^{*}_{Ne}$ radius for the function $F$ is atleast $\sigma_{0}$.
		
		To verify the sharpness of $\sigma_{0}$, take $f(z)=z/(1-z)^{-2\alpha+2}\in \mathcal{ST}(\alpha)$ and the polynomial $Q$ as $Q(z)=(1+z)^{n}$. Using these choices for the function $f$ and the polynomial $Q$, the expression for $zF'(z)/F(z)$ as in the proof for Theorem~\ref{thm2.1} is,
		\begin{align}
			\dfrac{zF'(z)}{F(z)}&= \dfrac{(1-2\alpha-\beta)z^2+(2-2\alpha+\beta)z+1}{1-z^2}.\label{eqn7.3}
		\end{align}
		The polynomial $\phi$ in \eqref{eqn7.1} provides that for $r=\sigma_{0}$,
		\begin{equation}\label{eqn7.4}
			3((1-2\alpha-\beta)r^2-(2-2\alpha+\beta)r+1)=(1-r^2).
		\end{equation}
		Thus using \eqref{eqn7.4}, and replacing $z=-\sigma_{0}$, \eqref{eqn7.3} gives $((-\sigma_{0})F'(-\sigma_{0}))/F(-\sigma_{0})=1/3=\varphi_{Ne}(-1)$. This proves the sharpness for $\sigma_{0}$.
		
		Case (ii): $2\alpha+\beta -2<0$. Here, it is known that $a_{F}(r)>1$, so, Lemma~\ref{lem7.1} directly gives that the required radius is atleast the positive solution of the equation $c_{F}(r)=5/3-a_{F}(r)$ that is   $\tilde{\sigma_{0}}$.
		
		To verify sharpness in this case, take $f(z)=z/(1-z)^{-2\alpha+2}\in \mathcal{ST}(\alpha)$ and the polynomial $Q$ as $Q(z)=(1+z)^{n}$. These  expressions transform $zF'(z)/F(z)$ into the form given in \eqref{eqn7.3}. The polynomial $\psi$ given in \eqref{eqn7.2} implies that for $r=\tilde{\sigma_{0}}$,
		\begin{equation}\label{eqn7.5}
			3((1-2\alpha-\beta)r^2+(2-2\alpha+\beta)r+1)=5(1-r^2).
		\end{equation}
		Thus, the radius $\tilde{\sigma_{0}}$ is the best possible since an application of \eqref{eqn7.5} in \eqref{eqn7.3} gives that the expression for $zF'(z)/F(z)$ takes the value $5/3=\varphi_{Ne}(1)$ for $z=\tilde{\sigma_{0}}$. Thus, proving the sharpness for the radius $\tilde{\sigma_{0}}$.
	\end{proof}
	
	When $\alpha$ goes to $1$ in Theorem~\ref{thm7.2} we get:
	The radius of starlikeness associated with a nephroid domain for the function $F(z)=z(Q(z))^{\beta/n}$ where $Q$ is a non-constant polynomial of degree $n$ non-vanishing on the unit disc and $\beta >0$, comes out to be $2/(2+3\beta)$. 	
	When $\beta\to 0 $  in Theorem~\ref{thm7.2}, we obtain the radius of starlikeness associated with a nephroid domain for the class of starlike functions of order $\alpha, 0\leqslant \alpha <1$ obtained by  Wani\ and\  Swaminathan  \cite[Theorem~3.1(ii)]{LatSwami2}  when $A=1-2\alpha\ \text{and}\  B=-1$.
	
	\section{Starlike functions associated with modified sigmoid function}
	In 2020, Goel and Kumar \cite{PriyankaSivaSig} introduced the class $\mathcal{S}^{*}_{SG}$ of functions mapping $\mathbb{D}$ onto the domain $\Delta_{SG}=\{w: |\log w/(2-w)|< 1\}$, with $\mathcal{S}^{*}_{SG}=\mathcal{S}^{*}(2/(1+\mathit{e}^{-z}))$. They also proved the following lemma.
	\begin{lemma}\cite[lemma~2.2]{PriyankaSivaSig} \label{lem8.1}
		Let $2/(1+\mathit{e})<a<2\mathit{e}/(1+\mathit{e})$. If
		\[r_a=\dfrac{\mathit{e}-1}{\mathit{e}+1}-|a-1|,\]
		then $\{w: |w-a|<r_a\} \subset \Delta_{SG}$.
	\end{lemma}
	In the next result, we find the radius of starlikeness associated with modified sigmoid function for the function $F$ defined in \eqref{eqn2.1}.
	\begin{theorem}\label{thm8.2}
		If the function $f\in \mathcal{ST}(\alpha)$, then the $\mathcal{S}^{*}_{SG}$ radius for the function $F$ given in \eqref{eqn2.1} is given by
		\begin{align*}
			R_{\mathcal{S}^{*}_{SG}}(F)&= \begin{dcases}
				\sigma_{0} &\text{ if }\ 2\alpha+\beta-2\geqslant 0\\
				\tilde{\sigma_{0}} &\text{ if }\ 2\alpha+\beta-2<0,
			\end{dcases}
		\end{align*}
		where
		\begin{align*}
			\sigma_{0}&=2(\mathit{e}-1)\big((2-2\alpha+\beta)(\mathit{e}+1)  \\&\quad \nonumber+\sqrt{((\mathit{e}+1)(2-2\alpha+\beta))^2
				-4(\mathit{e}-1)(3+\mathit{e}-2\alpha-2\alpha\mathit{e}-\beta-\beta\mathit{e})}\big)^{-1},\intertext{and}
			\tilde{\sigma_{0}} &=2(\mathit{e}-1)\big((2-2\alpha+\beta)(\mathit{e}+1)\\
			&\quad \nonumber +\sqrt{((\mathit{e}+1)(2-2\alpha+\beta))^2
				+4(\mathit{e}-1)(1+3\mathit{e}-2\alpha-2\alpha\mathit{e}-\beta-\beta\mathit{e})}\big)^{-1}.
		\end{align*}
	\end{theorem}
	\begin{proof}
		We will show that $\mathbb{D}(a_F (r);c_F(r)) \subset \Delta_{SG}\ \text{for}\  0<r\leqslant R_{\mathcal{S}^{*}_{SG}}(F)$.
		
		Case (i): $2\alpha+\beta-2 \geqslant 0$. It is known that in this case, the center $a_{F}$ satisfies the inequality $a_{F}(r)\leqslant 1$. Consequently, the equation $c_{F}(r)= ((\mathit{e}-1)/(\mathit{e}+1))-|a_{F}(r)-1|$ becomes $c_{F}(r)=((\mathit{e}-1)/(\mathit{e}+1))-1+a_{F}(r)$. This equation is simplified into the form $\phi(r)=0$, with
		\begin{equation}\label{eqn8.1}
			\phi(r):= ((2-2\alpha-\beta)(\mathit{e}+1)-(\mathit{e}-1))r^2-(2-2\alpha+\beta)(\mathit{e}+1)r+(\mathit{e}-1).
		\end{equation}
		The polynomial $\phi$ has a zero in the interval $(0,1)$ since, $\phi(0)=(\mathit{e}-1)>0,\ \text{and}\ \phi(1)=-2\beta(\mathit{e}+1)<0$, and this positive number is 
		\begin{align*}
			\sigma_{0}&=2(\mathit{e}-1)\big((2-2\alpha+\beta)(\mathit{e}+1)  \\&\quad \nonumber+\sqrt{((\mathit{e}+1)(2-2\alpha+\beta))^2
				-4(\mathit{e}-1)(3+\mathit{e}-2\alpha-2\alpha\mathit{e}-\beta-\beta\mathit{e})}\big)^{-1}.
		\end{align*}
		Now, to verify the sharpness of the radius $\sigma_{0}$, we take $f(z)=z/(1-z)^{-2\alpha+2}\in \mathcal{ST}(\alpha)$ and polynomial $Q$ as $Q(z)=(1+z)^{n}$. Using these choices for the function $f$ and the polynomial $Q$, the expression for $zF'(z)/F(z)$ as seen in the proof for Theorem~\ref{thm2.1} is obtained to be
		\begin{align}
			\dfrac{zF'(z)}{F(z)}=&\label{eqn8.2} \dfrac{(1-2\alpha-\beta)z^2+(2-2\alpha+\beta)z+1}{1-z^2},
			\intertext{which implies that}
			2-\dfrac{zF'(z)}{F(z)} =&\label{eqn8.3} \dfrac{(-3+2\alpha+\beta)z^2-(2-2\alpha+\beta)z+1}{1-z^2}.
			\intertext{Thus, \eqref{eqn8.2} and \eqref{eqn8.3} give that}
			\left(\dfrac{zF'(z)}{F(z)}\right)\left( 2-\dfrac{zF'(z)}{F(z)}\right)^{-1} =&\label{eqn8.4}\dfrac{(1-2\alpha-\beta)z^2+(2-2\alpha+\beta)z+1}{(-3+2\alpha+\beta)z^2-(2-2\alpha+\beta)z+1}.
		\end{align}
		Further, the polynomial $\phi$ in \eqref{eqn8.1} gets reduced to
		\begin{equation}\label{eqn8.5}
			((1-2\alpha-\beta)r^2-(2-2\alpha+\beta)r+1)\mathit{e}=(-3+2\alpha+\beta)r^2+(2-2\alpha+\beta)r+1.
		\end{equation}
		for $r=\sigma_{0}$. Thus using \eqref{eqn8.5} in \eqref{eqn8.4} for $z=-\sigma_{0}$, it is seen that \[\left|\log\left( \left(\dfrac{zF'(z)}{F(z)}\right)\left( 2-\dfrac{zF'(z)}{F(z)}\right)^{-1}\right)\right|=\left|\log \dfrac{1}{\mathit{e}}\right|=1,\] thus proving the sharpness for $\sigma_{0}$.

		Case (ii): $2\alpha+\beta-2<0$. Here, the center $a_{F}(r)>1$, which implies that the equation $c_{F}(r)= ((\mathit{e}-1)/(\mathit{e}+1))-|a_{F}(r)-1|$ converts to $c_{F}(r)=((\mathit{e}-1)/(\mathit{e}+1))+1-a_{F}(r)$.  This is equivalent to $\psi(r)=0$ for
		\begin{equation}\label{eqn8.6}
			\psi(r):= ((2-2\alpha-\beta)(\mathit{e}+1)+(\mathit{e}-1))r^2+(2-2\alpha+\beta)(\mathit{e}+1)r-(\mathit{e}-1).
		\end{equation}
		The number
		\begin{align*}
			\tilde{\sigma_{0}} &=2(\mathit{e}-1)\big((2-2\alpha+\beta)(\mathit{e}+1)\\
			&\quad \nonumber+\sqrt{((\mathit{e}+1)(2-2\alpha+\beta))^2
				+4(\mathit{e}-1)(1+3\mathit{e}-2\alpha-2\alpha\mathit{e}-\beta-\beta\mathit{e})}\big)^{-1}
		\end{align*} 
		is the smallest positive zero of the polynomial $\psi$, and the observations $\psi(0)=-(\mathit{e}-1)<0,\ \text{and}\ \psi(1)=4(1-\alpha)(\mathit{e}+1)>0$ justify the existence of a zero in $(0,1)$.
		
		To verify that the radius $\tilde{\sigma_{0}}$ is the best possible, take the values of the function $f$ and the polynomial $Q$ same as in Case (i), thus the expression for $(zF'(z)/F(z))/(2-(zf'(z)/F(z)))$ is same as given in \eqref{eqn8.4}. The polynomial $\psi$ in \eqref{eqn8.6} gives
		\begin{equation}\label{eqn8.7}
			(1-2\alpha-\beta)r^2+(2-2\alpha+\beta)r+1=((-3+2\alpha+\beta)r^2-(2-2\alpha+\beta)r+1)\mathit{e}.
		\end{equation}
		for $r=\tilde{\sigma_{0}}$. Thus, putting $z=\tilde{\sigma_{0}}$ in \eqref{eqn8.4} and then using \eqref{eqn8.7}, it is obtained that
		\[\left|\log\left( \left( \dfrac{\tilde{\sigma_{0}}F'(\tilde{\sigma_{0}})}{F(\tilde{\sigma_{0}})}\right)
		\left(2-\dfrac{\tilde{\sigma_{0}}F'(\tilde{\sigma_{0}})}{F(\tilde{\sigma_{0}})}\right)^{-1}\right) \right|=\left|\log \mathit{e}\right|=1.\] This proves sharpness for $\tilde{\sigma_{0}}$.
	\end{proof}
	When $\alpha$ goes to $1$ in Theorem~\ref{thm8.2} we get that the radius of starlikeness associated with modified sigmoid function  for the function $F(z)=z(Q(z))^{\beta/n}$ where $Q$ is a non-constant polynomial of degree $n$ non-vanishing on the unit disc and $\beta >0$, is $(\mathit{e}-1)/(\mathit{e}-1+\beta(\mathit{e}+1))$. When $\beta\to 0$  in Theorem~\ref{thm8.2}, we obtain the radius of starlikeness associated with the modified sigmoid function for the class of starlike functions of order $\alpha$, $0\leqslant \alpha <1$. The radius of parabolic starlikeness and the radii of other related starlikeness including the one related to the lemniscate of Bernoulli can be investigated.

\end{document}